\documentclass[a4paper,10pt]{article}
\usepackage{amsmath, amsthm, amssymb}
\usepackage{url}

\usepackage{physics}
\usepackage{graphicx,lipsum}
\graphicspath{ {./downloads/} }

\usepackage[colorlinks,citecolor=red,urlcolor=blue,bookmarks=false,hypertexnames=true]{hyperref}
\usepackage{tikz}
\usetikzlibrary{calc}
\usetikzlibrary{shapes}
\usepackage[autostyle]{csquotes}
\makeatletter

\usepackage[colorlinks]{hyperref}
\usepackage[nameinlink,capitalize]{cleveref}
\newtheorem{theorem}{Theorem}[section]
\newtheorem{corollary}{Corollary}[theorem]

 \newtheorem{proposition}{Proposition}[section]
\newtheoremstyle{named}{}{}{\itshape}{}{\bfseries}{.}{.5em}{\thmnote{#3's }#1}
\theoremstyle{named}

\newtheorem{remark}{Remark}
\theoremstyle{definition}

\newtheorem{example}{Example}

\theoremstyle{remark}

\theoremstyle{conclusion}

\theoremstyle{observation}

\usepackage{xspace}
\usepackage[margin=1.0in]{geometry}

\usepackage{titlesec}

\usepackage{mathtools}

\DeclarePairedDelimiterX{\inp}[2]{\langle}{\rangle}{#1, #2}
\titleformat{\chapter}
  {\Large\bfseries}
  {}
  {0pt}
  {\huge}

\begin{document}
\begin{center}
\fontsize{13pt}{10pt}\selectfont
    \textsc{\textbf{ A Wide Class of Contact-Complex Riemannian Submersions}}
    \end{center}
\vspace{0.1cm}
\begin{center}
   \fontsize{12pt}{10pt}\selectfont
    \textsc{Maria Falcitelli and Cem Sayar}
\end{center}
\vspace{0.2cm}
\begin{center}
   \fontsize{12pt}{10pt}\selectfont
    \textsc{Abstract}
\end{center}
Locally conformal almost quasi-Sasakian manifolds set up a wide class of almost contact metric manifolds containing several interesting subclasses. Contact-complex Riemannian submersions whose total space belongs to each of the considered classes are studied. In particular, the Chinea-Gonzalez class of the fibres and the Gray-Hervella class of the base space are determined. The main properties of the O'Neill invariants are stated. This allows discussing the integrability of the horizontal distribution and the minimality of the fibres.
\section{Introduction}

The theory of Riemannian submersions whose total space admits an almost contact metric structure has been developed in the last three decades \cite{Chi}, \cite{Wat}, \cite{Fal}, \cite{Fa2}, \cite{Tshi}, \cite{BT}. It is known that one can consider two types of such submersions, namely contact metric submersions (almost contact metric submersions of type I) and contact-complex Riemannian submersions (almost contact metric submersions of type II). In particular, we recall that a contact-complex Riemannian submersion is a Riemannian submersion $\pi : M \rightarrow B$ from an almost contact metric manifold $(M,\varphi , \xi , \eta , g)$ onto an almost Hermitian manifold $(B,J,g')$ such that $J \circ \pi_* = \pi_* \circ \varphi.$\\

Contact-complex submersions are relevant in certain physical theories, mainly if Hamiltonian dynamics and symplectic geometry are involved, \cite{Hamilton}.

\begin{itemize}
\item In Hamiltonian mechanics, contact structures are pivotal in studying symplectic geometry. It is known that Hamiltonian dynamics describes the evolution of physical systems and can be formulated by using symplectic geometry. Contact-complex submersions have the potential to understand the dynamics of specific systems.
    \item In the study of certain field theories and gauge theories, differential geometry concepts, including almost contact structures, can be employed. The geometric structures provided by contact complex submersions may be useful in describing the geometry of the configuration space of fields.

\item Geometric quantization is an approach to quantizing classical mechanics using geometric structures on phase spaces. Symplectic and contact structures are central to this theory. The choice of contact complex submersions have a potential to play a role in defining a suitable quantization scheme for specific systems.
\end{itemize}

In this paper we develop a detailed study of such submersions, assuming that the total space $M^{2m+1}, (m\geq 2)$ is a locally conformal (l.c.) almost quasi-Sasakian manifold. The class of these manifolds, which is closely related to the Chinea and Gonzalez-D\'{a}vila classification of almost contact metric manifolds, contains several interesting subclasses, such as locally conformal cosymplectic, almost cosymplectic, $\alpha-$Kenmotsu, $\alpha-$Sasakian, and quasi-Sasakian manifolds. Contact metric submersions whose total space is an l.c. almost quasi-Sasakian manifolds are studied in \cite{Fa2}.

Firstly we establish the main properties of the O'Neill invariants. This helps in determining the Chinea and Gonzalez-D\'{a}vila class of the fibres, which inherit from the total space an almost contact metric structure, and the Gray-Hervella class of the base space. One also proves that the 1-form $\eta$ is co-closed. It follows that contact-complex Riemannian submersions from an $\alpha$-Kenmotsu manifold, $\alpha \neq 0$, cannot exist. \\
Then, we consider submersions such that the Lee form of the total space vanishes, namely the fundamental form is closed. New results are stated. In particular, the horizontal distribution is totally geodesic, if the total space is almost cosymplectic or it is a $C_{12}$-manifold. The submersion has minimal fibres, if the total space is almost cosymplectic or quasi-Sasakian. We give a method useful to obtain explicit examples. In particular, if the structure of the total space is $\alpha$-Sasakian, $\alpha \neq 0$, the horizontal distribution cannot be integrable, the fibres are minimal and the base space is a K\"{a}hler manifold.\\
The last section deals with submersions whose total space is an l.c. (almost) cosymplectic manifold. In this case, the O'Neill integrability tensor depends on the vertical component of the Lee vector field and the mean curvature of the fibres is represented by the horizontal component.\\
Explicit examples are the canonical projections of warped product manifolds $M' \times_pF$ onto $M'$, $M'$ being a K\"{a}hler manifold of (real) dimension two, and $F$ an almost cosymplectic (cosymplectic) manifold. In fact, for any non-constant smooth function $f: M' \rightarrow \mathbb{R}, \, f>0$, we define a conformal almost cosymplectic (conformal cosymplectic) structure on the manifold $M' \times_pF$ which makes the projection $\pi : M' \times_pF \rightarrow M'$ a contact-complex Riemannian submersion.

\section{Fundamental properties of the O'Neill invariants and transference of the structure}\label{section1}
We recall some definitions and results on almost contact metric manifolds (ACM-manifolds) and almost Hermitian manifolds (AH-manifolds). \\

Let $(M,\varphi, \xi, \eta, g) $ be an ACM-manifold with $\dim M =2m+1$ and denote by $\nabla$ the Levi-Civita connection of $(M,g)$ and by $\phi$ the fundamental form,
\begin{equation}\label{eq1}
\phi(X,Y)=g(X,\varphi Y).
\end{equation}
The Lee form of $M$ is the 1-form $\omega$ defined by
\begin{equation}\label{eq2}
  \omega = -\frac{1}{2(m-1)}(\delta \phi \circ \varphi + \nabla_{\xi} \eta)+\frac{\delta \eta}{2m} \eta, \quad if\quad m\geq 2,
  \end{equation}
\begin{equation}\label{eq3}
\omega = \nabla_{\xi}\eta +\frac{\delta \eta}{2}\eta, \quad if \quad m=1,
\end{equation}
where $\delta$ denotes the co-differential operator.

According to the theory developed in \cite{Chi-Gon}, we denote $\tau_i$, $i \in \{1,...,12\}$, the $C_i-component$ of $\nabla \phi$. In particular, one has
\begin{eqnarray*}
  \delta \eta &=& \bar{c}(\tau_5)(\xi),\\
  \delta \phi (\varphi X)&=&-c(\tau_4)(\varphi X)+\tau_{12}(\xi, \xi ,\varphi X),\\
  (\nabla_\xi \eta)X &=&\tau_{12}(\xi,\xi,\varphi X),
\end{eqnarray*}
where, for any section $\alpha$ of the vector bundle $C(M)=\bigoplus_{1\leq i \leq 12} C_i(M)$, $c(\alpha)$ and $\bar{c}(\alpha)$ are the 1-forms acting as
\begin{eqnarray*}
c(\alpha)(X)\mid_{U}&=&\sum_{1 \leq i \leq 2m+1}\alpha(e_i,e_i,X),\\
\bar{c}(\alpha)(X)\mid_{U} &=&\sum_{1 \leq i \leq 2m+1} \alpha(e_i, \varphi e_i ,X),
\end{eqnarray*}
$\{e_1,...,e_{2m+1}\}$ being a local orthonormal frame on an open set $U$.
Therefore, the Lee form satisfies
\begin{equation}\label{eq4}
  \omega(X)=\frac{1}{2(m-1)}c(\tau_4)(\varphi X)+\frac{\bar{c}(\tau_5)(\xi)}{2m}\eta(X), \quad if \quad m \geq 2,
\end{equation}
\begin{equation}\label{eq5}
  \omega(X)=\tau_{12}(\xi,\xi,\varphi X)+\frac{\bar{c}(\tau_5)(\xi)}{2}\eta(X), \quad if \quad m=1.
\end{equation}
Through the paper we denote by $B=\omega^{\sharp}$ the Lee vector field.
Locally conformal almost quasi-Sasakian manifolds, introduced by I. Vaisman in \cite{Va}, are characterized as the ACM-manifolds $(M,\varphi, \xi, \eta, g) $ admitting a closed 1-form $\omega$ such that
\begin{equation}\label{eq7}
  d\phi =-2\omega \wedge \phi.
\end{equation}
Moreover, if $\dim M \geq 5$ and there exists a 1-form $\omega$ satisfying \eqref{eq7}, then $\omega$ is closed and it is the Lee form of $M$. We recall another characterization of these manifolds.
\begin{theorem}\cite{Fa2}\label{thm1}
Let $(M,\varphi, \xi, \eta, g)$ be an ACM-manifold, $\dim M \geq 5$. The following conditions are equivalent:
\begin{itemize}
  \item[\textbf{i) }] $d\phi=-2 \omega \wedge \phi,$
  \item[\textbf{ii) }] $M$ is a $C_2 \oplus \bigoplus_{4\leq h \leq 7} C_h \oplus \bigoplus_{9\leq h \leq 12}C_h$-manifold and the projections $\tau_{10}$ and $\tau_{11}$ are related by $2 \tau_{10}(X,Y,\xi)+\tau_{11}(\xi, X, Y)=0.$
\end{itemize}
\end{theorem}
By \eqref{eq4} and Theorem \ref{thm1}, it follows that $\phi $ is closed if and only if $M$ is in the class $C_2 \oplus C_6 \oplus C_7 \oplus \bigoplus_{9\leq h \leq 12}C_h$ and $2 \tau_{10}(X,Y,\xi)+\tau_{11}(\xi, X, Y)=0.$

Let $(M',J,g')$ be and AH-manifold, $\dim M'=2n$ and denote by $\Omega$ the fundamental form,
\begin{equation}\label{eq8}
  \Omega(X,Y)=g'(X,JY).
\end{equation}
If $n \geq 2$ we call Lee form of $M'$ the 1-form
\begin{equation}\label{eq9}
  \beta = -\frac{1}{2(n-1)}\delta \Omega \circ J.
\end{equation}
We remark that $\omega'=-2\beta$ is named the Lee form of $(M',J,g')$ by I. Vaisman \cite{Va2}.

If $\dim M'\geq 6$, then $M'$ falls in the class $W_2 \oplus W_4$ if and only if there exists a $1-form$ $\beta$ such that
\begin{equation}\label{eq11}
  d\Omega=-2\beta \wedge \Omega.
\end{equation}
Moreover, \eqref{eq11} entails that $\beta$ is the Lee form and $\beta$ is closed.

In this section, we fix a contact-complex Riemannian submersion (contact-complex R. submersion) $\pi : (M, \varphi, \xi, \eta, g) \rightarrow (M', J, g')$, $\dim M=2m+1 \geq 5$, $\dim M' =2n$ and put $m=n+r$.
If $r=0$, any fibre $F$ of $\pi$ has dimension 1, $TF$ is spanned by $\xi\mid_F$. If $r \geq 1$, then any fibre
$(F,\,\, \hat{\varphi}=\varphi\mid_{TF},\,\, \hat{\xi}=\xi\mid_F, \,\,\hat{\eta} =\eta\mid_{TF},\,\, \hat{g}=g\mid_{TF\times TF})$
is an ACM-manifold.

We denote by $v$ the projection on the vertical distribution $\mathcal{V}=ker\pi_*$ and by $h$ the projection onto the horizontal distribution $\mathcal{H}.$ The O'Neill invariants $T$ and $A$ of $\pi$ are defined as:
\begin{eqnarray*}
T_EF&=&h(\nabla_{vE}vF)+v(\nabla_{vE}hF),\\
A_EF&=&v(\nabla_{hE}hF)+v(\nabla_{hE}vF),
\end{eqnarray*}
for any vector fields $E,F$ on $M$.

We refer to \cite{O}, \cite{Fal} for details on Riemannian submersions.
\begin{proposition}\label{Prop1}
Let $\pi: M \rightarrow M'$ be a contact-complex Riemannian submersion. Assume that $d\phi=-2\omega \wedge \phi$. Then, one has:
\begin{enumerate}
  \item $\delta \eta =0$,
  \item The invariant $A$ of $\pi$ satisfies:
  \begin{itemize}
    \item $A_XU=\eta(U)A_X \xi-\omega (\varphi U)\varphi X$,
    \item $A_XY=\eta(A_XY)\xi+\phi (X,Y)v (\varphi B)$,
    \item $A_{\varphi X}U-\varphi(A_XU)=\eta(U)(A_{\varphi X}\xi-\varphi (A_X\xi))$,
    \item $A_XY-A_{\varphi X}\varphi Y=\eta(A_XY-A_{\varphi X}\varphi Y)\xi$
  \end{itemize}
  for any $X,Y \in \chi^{h}(M)$, $U \in \chi^{v}(M)$.
\end{enumerate}
\end{proposition}
\begin{proof}
  Let $X,Y$ be basic vector fields such that $\pi_*(X)=X'$, $\pi_*(Y)=Y'$ and $U$ be vertical. Then $[X,U], [Y,U]$ are vertical and
  \begin{equation*}
    2A_XY=v([X,Y]).
  \end{equation*}
It follows that
\begin{eqnarray*}
3d\phi(X,Y,U)&=&U(\phi(X,Y))+X( \phi(Y,U))+Y(\phi(U,X))\\
& & -\phi([X,Y],U)-\phi([Y,U],X)-\phi([U,X],Y)\\
&=& U(\phi(X,Y))-\phi([X,Y],U)\\
&=& \underbrace{\pi_*U}_{=0}(g'(X'JY'))-g([X,Y],\varphi U)\\
&=& -2g(A_XY,\varphi U).
\end{eqnarray*}
Being
\begin{equation*}
  3 d \phi (X,Y,U)=-2(\omega(X)\underbrace{\phi(Y,U)}_{=0}+\omega(Y)\underbrace{\phi(U,X)}_{=0}+\omega(U)\phi(X,Y))
\end{equation*}
one has
\begin{equation}\label{eq10}
  g(A_XY,\varphi U)=\omega(U)g(X,\varphi Y).
\end{equation}
By putting $U=\xi$ in \eqref{eq10} one has $\omega(\xi)=0$ and 1. is proved.
By \eqref{eq10}, we also have
\begin{equation*}
g(A_X\varphi U,Y)=\omega(U)g(\varphi X,Y), \,Y \,\,\text{basic}.
\end{equation*}
Since basic vector fields locally span the horizontal distribution, we obtain
\begin{eqnarray*}
 A_X \varphi U &=& \omega(U) \varphi X,\,U\in \chi^{v}(M).
 \end{eqnarray*}
 It follows that
 \begin{equation*}
  A_X\varphi^2 U = \omega(\varphi U) \varphi X.
\end{equation*}
So, we obtain the first relation in 2., for any basic $X$ and then for any horizontal vector field $X$. The other relations are an easy consequence of the stated formulas.
\end{proof}

\begin{remark}
Since the $C_5$-component of $\nabla \phi$ is determined by $\delta \eta,$
by Proposition \ref{Prop1} one has that if $(M,\varphi, \xi, \eta, g) $ is the total space of $\pi$ and $d\phi =-2\omega \wedge \phi$, then $M$ falls in the class
\begin{equation*}
  C_2 \oplus C_4 \oplus C_6 \oplus C_7 \oplus \bigoplus_{9\leq h \leq 12}C_h \,\text{and}\, \, 2\tau_{10}(X,Y,\xi)+\tau_{11}(\xi, X, Y)=0.
\end{equation*}
\end{remark}
In particular we have: $d\phi=0$ ($M$ is almost quasi-Sasakian) if and only if $M$ is in $C_2 \oplus C_6 \oplus C_7 \oplus \bigoplus_{9\leq h \leq 12}C_h \,\text{and} \,\, 2\tau_{10}(X,Y,\xi)+\tau_{11}(\xi, X, Y)=0.$

In this case, by Proposition \ref{Prop1}, the invariant $A$ satisfies
\begin{equation*}
  A_XU=\eta(U)A_X\xi, \quad A_XY=\eta(A_XY)\xi,
\end{equation*}
so $A_X\xi, X\in \chi^h(M), $ determines $A.$
\begin{proposition}\label{Prop2}
Let $(F,\hat{\varphi}, \hat{\xi}, \hat{\eta}, \hat{g})$ be a leaf of the contact-complex Riemannian submersion $\pi: M \rightarrow M'$, $\dim F=2r+1\geq 3$. Assume that $d\phi=-2\omega \wedge \phi.$ Then, we have
\begin{enumerate}
  \item $\hat{\delta }\hat{\eta}$=0,
  \item If $r\geq 2$, the fundamental form $\hat{\phi}$ of $F$ satisfies $d\hat{\phi}=-2\hat{\omega} \wedge \hat{\phi}$ and $\hat{\omega}=\omega\mid_{TF}$ is the Lee form of $F$. \\
      If $r=1$ the Lee form $\hat{\omega}$ of $F$ acts as $\hat{\omega}(U)=(\nabla_\xi\eta)U$, for any $U$ tangent to $F$.
\end{enumerate}
\end{proposition}
\begin{proof}
  We consider a local orthonormal frame on $M$ \[\{e_1, ... , e_{2n}, v_1,...,v_{2r},\xi\}\]
  such that $\{v_1,...,v_{2r}\}$ are tangent to $F.$ In this case, $\{e_1, ... , e_{2n}\}$ are horizontal. We have
  \begin{eqnarray*}
\delta \eta &=& -\sum_{i=1}^{2n}g(\nabla_{e_i}\xi,e_i)-\sum_{i=1}^{2r}g(\nabla_{v_i}\xi,v_i)\\
&=&\sum_{i=1}^{2n}g(\underbrace{A_{e_i}e_i}_{=0},\xi)-\sum_{i=1}^{2r}\hat{g}(\hat{\nabla}_{v_i}\hat{\xi},v_i)\\
&=& \hat{\delta }\hat{\eta}.
  \end{eqnarray*}
  So, 1. follows from Proposition \ref{Prop1}.

Now, assume $r\geq 2$ and, applying Proposition \ref{Prop1}, for every $U \in TF$ we have
\begin{eqnarray*}
\omega(U)&=& -\frac{1}{2(m-1)}(\delta \phi(\varphi U)+(\nabla_\xi\eta)U)\\
&=& -\frac{1}{2(m-1)}\bigg( \sum_{i=1}^{2n}g((\nabla_{e_i}\varphi)e_i,\varphi U)+   \sum_{i=1}^{2r} g((\nabla_{v_i}\varphi)v_i,\varphi U)\bigg)\\
&=& -\frac{1}{2(m-1)}\bigg(\sum_{i=1}^{2n}g(A_{e_i}\varphi e_i,\varphi U)+\sum_{i=1}^{2r} g((\hat{\nabla}_{v_i}\hat{\varphi})v_i,\hat{\varphi} U)\bigg)\\
&=& -\frac{1}{2(m-1)}\bigg(\sum_{i=1}^{2n}\phi(e_i,\varphi e_i)g(\varphi B, \varphi U)-2(r-1)\hat{\omega}(U)\bigg)\\
&=& \frac{1}{m-1}(n \omega(U)+(r-1)\hat{\omega}(U)).
\end{eqnarray*}
Being $m=n+r$, we have $\omega(U)=\hat{\omega}(U).$ It follows that $\hat{\omega}=\omega\mid_{TF}.$\\
The equality $d\hat{\phi}=-2\hat{\omega} \wedge \hat{\phi}$ follows from $d\phi=-2\omega \wedge \phi.$\\
Finally, if $\dim F=3$, since $\hat{\delta}\hat{\eta}=0$, we have $\hat{\omega}=\hat{\nabla}_{\hat{\xi}}\hat{\eta}$ and then for any $U \in TF$ $\hat{\omega}(U)=(\nabla_{\xi}\eta)(U).$
\end{proof}
\begin{proposition}\label{Prop3}
Let $\pi : M \rightarrow M'$ be a contact-complex Riemannian submersion such that $\dim M=2m+1 \geq 5$, $\dim M'=2n \geq 4$. Assume that $d\phi=-2\omega \wedge \phi$. Then, the Lee form $\beta$ of $M$ and $M'$ satisfies
\begin{equation}\label{eq12}
\beta(\pi_*(X))=\omega(X),
\end{equation}
for any basic vector field $X$. Moreover, $(M',J,g')$ is in the Gray-Hervella class $W_2 \oplus W_4$.
\end{proposition}
\begin{proof}
  Let $\{e_1, ..., e_{2n},u_1,...,u_r,\varphi u_1,...,\varphi u_r, \xi\}$ be a local orthonormal frame on $M$ such that $\{e_1, ..., e_{2n}\}$ are basic. Put $e_i'=\pi_*(e_i),\, i=1,...,2n,$ so that $\{e_1', ..., e_{2n}'\}$ is a local orthonormal frame on $M'.$

Given a basic vector field $X$, we put $X'=\pi_* X$. Applying the equality
\begin{equation}\label{new1}
  4(n-1)\beta(X')=3 \sum_{i=1}^{2n}d\omega(e_i',Je_i',X'),
\end{equation}
and the hypothesis $d\phi=-2\omega \wedge \phi$, we obtain
\begin{eqnarray*}
4(n-1)\beta(X')&=&3\sum_{i=1}^{2n}d\omega(e_i',Je_i',X')=3\sum_{i=1}^{2n}d\phi(e_i,\varphi e_i,X)\\
&=&-2\sum_{i=1}^{2n}\bigg(\omega(e_i)\phi(\varphi e_i, X)+\omega(\varphi e_i)\phi(X,e_i)\bigg)+4n\omega(X)\\
&=&-4\sum_{i=1}^{2n}\omega(e_i)g(e_i,X)+4n\omega(X)=4(n-1)\omega(X).
\end{eqnarray*}
Being $n\geq 2$, it follows that $\beta(X')=\omega(X)$.\\
Finally, given $X',Y',Z' \in \chi(M')$, let $X,Y,Z$ be their horizontal lifts. Then, we have
\begin{eqnarray*}
d\Omega(X',Y',Z')&=& d\phi(X,Y,Z)=-2(\omega \wedge \phi)(X,Y,Z)\\
&=& -2(\beta \wedge \Omega)(X',Y',Z').
\end{eqnarray*}
It follows that $d\Omega=-2\beta \wedge \Omega$ and $M'$ is a $W_2 \oplus W_4$-manifold.
\end{proof}
\begin{proposition}\label{Prop4}
Let $\pi: M \rightarrow M'$ be a contact-complex Riemannian submersion. Assume that $d\phi=-2\omega \wedge \phi.$ For any $X \in \chi^h(M)$, $U \in \chi^v(M)$ one has
\begin{equation}\label{eq15}
         h((\nabla_X \varphi)U)=\omega(U)\varphi X- \omega(\varphi U) X-\eta(U)\varphi(A_X\xi),
        \end{equation}
 \begin{equation}\label{eq16}
          h((\nabla_U\varphi)X)=\eta(U)(A_{\varphi X}\xi-\varphi(A_X\xi)).
        \end{equation}
\end{proposition}
\begin{proof}
  By Proposition \ref{Prop1}, given $X,Y \in \chi^h(M)$, $U \in \chi^v(M)$, one has
  \begin{eqnarray*}
g((\nabla_X\varphi)U,Y)&=&g(A_X\varphi U-\varphi(A_XU),Y)\\
&=&-\omega(\varphi^2U)g(\varphi X,Y)-g(\eta(U)\varphi(A_X\xi)-\omega(\varphi U)\varphi^2X,Y)\\
&=&g(\omega(U)\varphi X-\eta(U)\varphi(A_X\xi)-\omega(\varphi U)X,Y).
\end{eqnarray*}
  Thus, we obtain \eqref{eq15}. Moreover, if $X$ is basic, by Proposition \ref{Prop1} we get
\begin{eqnarray*}
h((\nabla_U \varphi)X)&=&h(\nabla_{U}\varphi X-\varphi (A_XU))\\
&=&A_{\varphi X}U-\varphi(A_XU)=\eta(U)(A_{\varphi X}\xi-\varphi (A_X\xi)).
\end{eqnarray*}
Then, one gets \eqref{eq16}, since basic vector fields locally span the horizontal distribution.
\end{proof}
The next result allows us to express the vertical components of $(\nabla_X \varphi)U,\,(\nabla_U \varphi)X$ in terms of the invariant $T$, the Lee form $\omega$ and the Nijenhuis tensor field $N_\varphi $ of $\varphi$, which is denoted by $N_\varphi$, acting as
\begin{equation}\label{eq17}
N_\varphi(X,Y)=\varphi^2[X,Y]+[\varphi X, \varphi Y]-\varphi [X, \varphi Y]-\varphi[\varphi X,Y],\,\, X,Y \in \chi(M).
\end{equation}
\begin{proposition}\label{Prop5}
Let $\pi: M \rightarrow M'$ be a contact-complex Riemannian submersion. Assume that $d\phi=-2\omega \wedge \phi.$ Then, we have
\begin{equation}\label{eq18}
  v((\nabla_X \varphi)U)=-\eta(U)\varphi (T_\xi X)-g(T_\xi X, \varphi U)\xi,\, X\in \chi^h(M),\, U \in \chi^v(M),
\end{equation}
\begin{eqnarray}\label{eq19}
  v((\nabla_U \varphi)X)&=&\eta(U)(T_\xi \varphi X-\varphi(T_\xi X))+\omega(X)\varphi U+\omega(\varphi X)\varphi^2 U \nonumber\\
  & &+\frac{1}{2}N_\varphi (X, \varphi U), \,\, X \in \chi^h(M),\, U \in \chi^v(M)
\end{eqnarray}
The invariant $T$ satisfies:
\begin{equation}\label{eq20}
T_U \varphi V - T_{\varphi U}V=\eta(U)T_{\varphi V} \xi -\eta(V)T_{\varphi U}\xi+2 \phi(U,V)h(B),\,U,V \in \chi^v(M).
\end{equation}
\end{proposition}
\begin{proof}
  Given $U,V \in \chi^v(M)$, $X \in \chi^h(M)$, $X$ basic, we have
  \begin{eqnarray*}
-2\omega(X)\phi(U,V)&=&3d\phi(U,V,X)\\
&=&-g((\nabla_U \varphi)V,X)+g((\nabla_V \varphi)U,X)-g((\nabla_X \varphi)U,V)\\
&=&-g(T_U \varphi V-T_V \varphi U,X)-g((\nabla_X \varphi)U,V)
  \end{eqnarray*}
  and then
  \begin{equation}\label{eq21}
    g((\nabla_X \varphi)U,V)=-g(T_U \varphi V-T_V \varphi U,X)+2\omega(X)\phi(U,V).
  \end{equation}
  In particular, we obtain:
  \begin{equation*}
    g((\nabla_X \varphi)\xi,V)=-g(T_\xi \varphi V, X)=g(T_\xi X, \varphi V).
  \end{equation*}
  This relation entails
  \begin{equation}\label{eq22}
    v(\nabla_X \xi)=T_\xi X- \eta(T_\xi X)\xi.
  \end{equation}
  Moreover, by direct calculation, we have
  \begin{eqnarray*}
g((\nabla_X \varphi)U,V)&=&g(\nabla_{\varphi U}X,V)+g(\nabla_U X,\varphi V)+g([X,\varphi U]-\varphi [X,U],V)\\
&=&g(T_{\varphi U}X,V)+g(T_UX,\varphi V)+g((\mathfrak{L}_X\varphi)U,V)\\
&=&-g(T_U\varphi V+T_{\varphi U}V,X)+g((\mathfrak{L}_X\varphi)U,V).
  \end{eqnarray*}
  Hence applying \eqref{eq21}, we have
  \begin{equation}\label{eq23}
    g(T_{\varphi U}X,V)+\omega(X)g(\varphi U,V)=-\frac{1}{2}g((\mathfrak{L}_X\varphi)U,V).
  \end{equation}
  We remark that, being $X$ is basic, $(\mathfrak{L}_X\varphi)U$ is vertical, so \eqref{eq23} implies
  \begin{equation}\label{eq24}
    T_{\varphi U}X+\omega(X)\varphi U=-\frac{1}{2}(\mathfrak{L}_X\varphi)U,\, U \in \chi^v(M),\,X \in \chi^b(M).
  \end{equation}
  Then, we also obtain
  \begin{equation}\label{eq25}
    T_{\varphi^2U}X+\omega(X)\varphi^2U=-\frac{1}{2}(\mathfrak{L}_X\varphi)\varphi U, \, U \in \chi^v(M),\,X \in \chi^b(M).
  \end{equation}
  Furthermore, a direct computation and \eqref{eq22} give
  \begin{eqnarray*}
(\mathfrak{L}_X\varphi)\varphi U+\varphi((\mathfrak{L}_X\varphi) U)&=&g(\nabla_X \xi,U)\xi+\eta(U)[X,\xi]+g(\nabla_U X, \xi)\xi\\
&=&(2g(T_\xi X,U)-\eta(T_\xi X)\eta(U))\xi+\eta(U)[X,\xi].
  \end{eqnarray*}
  So applying \eqref{eq24} and \eqref{eq25}, one has
  \begin{equation*}
T_{\varphi^2 U}X+\varphi(T_{\varphi U}X)+2\omega(X)\varphi^2U=-\frac{1}{2}(2g(T_\xi X,U)-\eta(T_\xi X)\eta(U))\xi -\frac{1}{2}\eta(U)[X,\xi].
  \end{equation*}
  Equivalently, we have
  \begin{eqnarray}\label{eq100}
T_UX-\varphi(T_{\varphi U}X)&=&-2\omega(X)\big(U-\eta(U)\xi\big)+\eta(U)T_\xi X\nonumber\\
& &+\frac{1}{2}(2g(T_\xi X,U)-\eta(T_\xi X)\eta(U))\xi+\frac{1}{2}\eta(U)[X,\xi].
  \end{eqnarray}
  We consider $V \in \chi^v(M)$. By \eqref{eq22}, \eqref{eq100}, one obtains
  \begin{eqnarray*}
g(T_U \varphi V-T_{\varphi U}V,X)&=&-g(T_UX,\varphi V)+g(\varphi(T_{\varphi U}X),\varphi V)-\eta(V)g(T_{\varphi U}\xi,X)\\
&=&2\omega(X)g(U,\varphi V)+\eta(U)g(T_\xi \varphi V,X)-\eta(V)g(T_\xi\varphi U,X).
  \end{eqnarray*}
  This relation implies \eqref{eq20} and applying \eqref{eq21}, we get
  \begin{equation*}
g((\nabla_X \varphi)U,V)=-\eta(U)g(\varphi (T_\xi X),V)-g(T_\xi X, \varphi U)\eta(V).
  \end{equation*}
Hence, we obtain \eqref{eq18}.

Finally, by \eqref{eq25}, for any $X \in \chi^b (M)$, $U \in \chi^v(M)$, we have
\begin{eqnarray*}
v((\nabla_U \varphi )X)&=&T_U \varphi X -\varphi(T_UX)\\
&=&\eta(U)(T_\xi \varphi X-\varphi (T_\xi X))+\omega(X)\varphi U+\omega(\varphi X)\varphi^2 U\\
& &+\frac{1}{2}\big((\mathfrak{L}_{\varphi X}\varphi)\varphi U -\varphi((\mathfrak{L}_X\varphi)\varphi U)\big)\\
&=& \eta(U)(T_\xi \varphi X-\varphi(T_\xi X))+\omega(X)\varphi U+\omega(\varphi X)\varphi^2 U+\frac{1}{2}N_\varphi(X, \varphi U).
\end{eqnarray*}
Since the horizontal distribution is locally spanned by basic vector fields, this formula implies \eqref{eq19}.
\end{proof}
\begin{corollary}\label{cor1}
\textit{Let $\pi: M \rightarrow M'$ be a contact-complex Riemannian submersion. Assume that $d\phi=-2\omega \wedge \phi.$ Then, one has}
\begin{equation}\label{eq26-1}
  N=2rh(B)+T_\xi \xi.
\end{equation}
\end{corollary}
\begin{proof}
  By Proposition \ref{Prop5} for any $U,V \in \chi^v(M)$, we have
  \begin{eqnarray}\label{eq26}
T_UV+T_{\varphi U}\varphi V&=&\eta(U)T_V\xi +\eta(V)T_U\xi -\eta(U)\eta(V)T_\xi \xi\nonumber\\
& &+2g(\varphi U,\varphi V)h(B).
  \end{eqnarray}
Given $p\in M$, we consider an orthonormal basis $ \{u_1,...,u_r,\varphi u_1,...,\varphi u_r,\xi\}$ of the vertical space $V_p$ and applying \eqref{eq26} we have, at $p$:
\begin{equation*}
N=\sum_{i=1}^{r}(T_{u_i}u_i+T_{\varphi u_i}\varphi u_i)+T_\xi \xi =2r h(B)+T_\xi \xi.
\end{equation*}
\end{proof}

\section{The case $\omega =0$}\label{section2}
We examine some consequences of the results stated in Section \ref{section1}, assuming that the total space is in a suitable subclass of  $C_2 \oplus  C_6 \oplus C_7 \oplus \bigoplus_{9\leq h \leq 12}C_h$.

First, we consider almost cosymplectic manifolds, whose defining condition is
\begin{equation}\label{eq50}
  d \phi =0, \, d \eta=0.
\end{equation}

These manifolds set up the class $C_2 \oplus C_9$, \cite{Chi-Gon}.

\begin{theorem}\label{theo1}
  Let $\pi : M \rightarrow M'$ be a contact-complex Riemannian submersion and assume that $M$ is an almost cosymplectic manifold.
  Then, one has:
  \begin{enumerate}
    \item The horizontal distribution is integrable and totally geodesic.
    \item Any fibre of $\pi$ is a minimal submanifold of $M$ and inherits from $M$ an almost cosymplectic structure.
        \item For any $X \in \chi^h(M)$, $U \in \chi^v(M)$ one has
        \begin{equation*}
          (\nabla_X \varphi)U=0,\quad (\nabla_U \varphi)X=\frac{1}{2}N_\varphi (X, \varphi U).
        \end{equation*}
        \item $(M',J,g')$ is an almost K$\ddot{a}$hler manifold.
  \end{enumerate}
\end{theorem}

\begin{proof}
  Let $X,Y$ be horizontal vector fields. Then, we have
  \begin{equation*}
    g(A_XY,\xi)=\frac{1}{2}g([X,Y],\xi)=-d\eta (X,Y)=0.
  \end{equation*}
  Then, by Proposition \ref{Prop1}-2., one gets
  \begin{equation*}
    A_XY=g(A_XY,\xi)\xi=0, \, \forall X, Y \in \chi^h(M).
  \end{equation*}
Combining with the well-known properties of $A$, we obtain $A=0$, i.e. $\mathcal{H}$ is integrable. Moreover \[\nabla_X Y= h(\nabla_X Y)\in \chi^h(M),\]
which proves 1.

Let $((F,\,\, \hat{\varphi}=\varphi \mid_{TF},\,\, \hat{\xi}=\xi \mid_F, \,\,\hat{\eta} =\eta \mid_{TF},\,\, \hat{g}=g\mid_{TF\times TF})$ be a leaf of $\pi$. Then, \[d\hat{\phi}=d\phi \mid_{TF}=0,\,\,\, d \hat{\eta}=d\eta \mid_{TF}=0.\]
Moreover, we recall that, being $M$ a $C_2\oplus C_9$-manifold, one has $\nabla_\xi \xi=0$, so $T_\xi \xi= h(\nabla_\xi \xi)=0$ and by Corollary \ref{cor1}, the mean curvature vector field of $F $ vanishes, so we obtain 2.

Now, we prove that the operator $T_\xi$ vanishes. Since $d\eta=0$, for any $X \in  \chi^h(M)$, $U \in \chi^v(M)$, we have
\begin{eqnarray*}
g(T_U\xi,X)&=&g(\nabla_U\xi,X)=g(\nabla_X \xi,U)=g(T_\xi X,U)=-g(T_U\xi, X)\\
&\Rightarrow&g(T_U\xi,X)=g(T_\xi U,X)=0,\,\forall X \in \chi^h(M)\\
&\Rightarrow& T_\xi U=0, \, \forall U \in \chi^v(M)\Rightarrow T_\xi =0.
\end{eqnarray*}
Hence, by Proposition \ref{Prop4} and Proposition \ref{Prop5}, we obtain 3.

Property 4. is trivial since
\[d\phi=0\Rightarrow d \Omega=0,\]
so $(M',J,g')$ is an almost K$\ddot{a}$hler manifold.
\end{proof}

\begin{example}\label{3.1}
Let $(M', J, g')$ be an almost Hermitian manifold, with $\dim M'=2n$, $(F, \hat{\varphi}, \hat{\xi}, \hat{\eta}, \hat{g})$ an almost contact metric manifold, $\dim F=2r+1$. On the product manifold $M=M'\times F$ one considers the almost contact metric structure $(\varphi, \xi, \eta, g)$ such that, for any $X, Y \in \chi(M')$, $U,V \in \chi(F)$
\begin{equation}\label{eq51}
  \varphi(X,U)=(JX,\hat{\varphi}U),\quad \eta(X,U)=\frac{r}{n+r}\hat{\eta}(U),\quad \xi = \frac{n+r}{r}\hat{\xi}
\end{equation}
\begin{equation}\label{eq52}
  g((X,U),(Y,V))=g'(X,Y)+(\frac{r}{n+r})^2\hat{g}(U,V).
\end{equation}
The canonical projection $\pi: M\rightarrow M'$ is a contact-complex Riemannian submersion and, for any $x \in M'$, the fibre $\pi^{-1}(x)=\{x\}\times F$ is identified with $F$. So, the vertical and horizontal distributions are identified with $TF$ and $TM'$, respectively.

By a direct calculation, one has
\begin{eqnarray*}
d\phi(X,Y,Z)&=&d \Omega (X,Y,Z), \,\,d\eta(X,Y)=0,\,\, X,Y,Z \in \chi(M')\\
d\phi(X,Y,U)&=& d\phi(X,U,V)=0,\,\,d\eta(X,U)=0, \,\,X,Y \in \chi(M'),U,V\in \chi(F)\\
d\phi(U,V,W)&=&\bigg(\frac{r}{n+r}\bigg)^2d\hat{\phi}(U,V,W), \,\, U,V,W \in \chi(F)\\
d\eta(U,V)&=&\bigg(\frac{r}{n+r}\bigg)d\hat{\eta}(U,V).
\end{eqnarray*}
It follows that $(M,\varphi, \xi, \eta, g)$ is almost cosymplectic if and only if $(M',J,g')$ is almost K$\ddot{a}$hler and $(F, \hat{\varphi},\hat{\xi}, \hat{\eta}, \hat{g})$ is almost cosymplectic.

Hence, considering $(M',J,g')$ in the Gray-Hervella class $W_2$ and $(F, \hat{\varphi}, \hat{\xi}, \hat{\eta}, \hat{g})$ in $C_2 \oplus C_9$, the projection $\pi : M \rightarrow M'$ is a contact-complex Riemannian submersion such that both the vertical and horizontal distributions are totally geodesic (i.e. $T=0,$ $A=0$).

We also remark that the $C_9$-component of $\nabla \phi$ is determined by the tensor field $\nabla \xi $ and for any $X \in \chi(M)$, $U \in \chi(F)$ one has
\begin{equation*}
  \nabla_X \xi=0, \quad h(\nabla_U \xi)=0, \quad v(\nabla_U \xi)=\bigg(\frac{n+r}{r}\bigg) \hat{\nabla}_U \hat{\xi}.
\end{equation*}
It follows that $M$ is a $C_2$-manifold if and only if $F$ is a $C_2$-manifold. Moreover, as stated in \cite{Tshi}, $M$ is a $C_9$-manifold if and only if $M'$ is K$\ddot{a}$hler and $F$ falls in $C_9.$
\end{example}

Now we consider the class $C_6\oplus C_7$, which consists of the quasi-Sasakian manifolds. The defining condition is
\begin{equation}\label{eq53}
d\phi=0,\quad N_\varphi+2d \eta \otimes \xi=0.
\end{equation}
It is known that any ACM-manifold $(M,\varphi, \xi, \eta, g)$ such that $\dim M \geq 5$ is quasi-Sasakian if and only if
\begin{equation}\label{eq54}
(\nabla_X \phi)(Y,Z)=\eta(Y)(\nabla_{\varphi X}\eta)Z+\eta(Z)(\nabla_Y\eta)\varphi X,
\end{equation}
for any $X,Y,Z \in \chi(M)$, \cite{Fa2}.

\begin{theorem}\label{theo2}
  Let $\pi: M\rightarrow M'$ be a contact-complex Riemannian submersion. If $M$ is quasi-Sasakian, then the following properties hold:
  \begin{enumerate}
    \item $(M',J,g')$ is a K$\ddot{a}$hler manifold.
    \item If $r \geq 1$, any fibre is a minimal submanifold of $M$ and inherits from $M$ a quasi-Sasakian structure.
        \item The invariant $A$ satisfies
        \[A_{\varphi X}U=\varphi(A_XU)=\eta(U)A_{\varphi X}\xi,\,X \in \chi^h(M),\, U \in \chi^v(M).\]
        \item The invariant $T$ satisfies
        \[T_\xi \circ \varphi =\varphi \circ T_\xi\]
        \[T_U \varphi V-\varphi(T_UV)=-\eta(V)T_\xi(\varphi U)\]
        \item For any $X \in \chi^h(M),\, U \in \chi^v(M)$, one has
        \[(\nabla_U \varphi)X=g(T_\xi \varphi U,X)\xi,\]
        \[(\nabla_X \varphi)U=-\eta(U)\bigg(\varphi(A_X \xi+T_\xi X)-g(\varphi(T_\xi X),U)\xi\bigg).\]
  \end{enumerate}
\end{theorem}

\begin{proof}
  Considering $X',Y',Z' \in \chi(M')$, let $X,Y,Z$ be their horizontal lifts. By \eqref{eq54}, we have
  \begin{equation*}
    (\nabla'_{X'}\Omega)(Y',Z')\circ \pi=(\nabla_X \phi)(Y,Z)=0.
  \end{equation*}
  Then, $(M',J,g')$ is a K$\ddot{a}$hler manifold.

  Let $(F, \hat{\varphi},\hat{\xi}, \hat{\eta}, \hat{g})$ be a fibre $(\dim F =2r+1 \geq 3).$ It is obvious to prove that $F$ is normal and has closed fundamental form. By Corollary \ref{cor1}, since $T_\xi \xi=h(\nabla_\xi \xi)=0$, we obtain that $(F,\hat{g})$ is a minimal submanifold of $M$.

  Given $X \in \chi^h(M)$, $U \in \chi^v(M)$, for any $Y \in \chi^h(M)$, we have
  \begin{equation*}
    g(A_X\varphi U-\varphi(A_XU),Y)=(\nabla_X \phi)(Y,U)=\eta(U)(\nabla_Y \eta)\varphi X=-\eta(U)g(A_{\varphi X}\xi,Y)
  \end{equation*}
It follows
\begin{equation}\label{eq55}
  A_X\varphi U-\varphi(A_XU)=-\eta(U)A_{\varphi X}\xi.
\end{equation}
In particular, one has
\begin{equation}\label{eq56}
  \varphi(A_X\xi)=A_{\varphi X} \xi.
\end{equation}
So, by Proposition \ref{Prop1}, one gets
\begin{equation*}
  A_{\varphi X}U=\varphi(A_XU)=\eta(U)A_{\varphi X}\xi.
\end{equation*}
Applying again \eqref{eq54} for any $V \in \chi^v(M)$, we have
\begin{equation}\label{eq57}
  g((\nabla_X\varphi)U,V)=-\eta(U)g(\nabla_{\varphi X}\xi,V)-\eta(V)g(T_U\xi,\varphi X).
\end{equation}
Putting in \eqref{eq57} $V=\xi$, we have
\begin{equation*}
  g(\nabla_X \xi,\varphi U)=g(T_\xi U,\varphi X).
\end{equation*}
This relation also implies
\begin{equation*}
  g(\nabla_{\varphi X}\xi,U)=g(T_\xi \varphi U,X)
\end{equation*}
and by \eqref{eq57},  we get
\begin{equation*}
  g((\nabla_X \varphi)U,V)=-\eta(U)g(T_\xi \varphi V,X)+\eta(V)g(\varphi(T_\xi U),X).
\end{equation*}
It follows
\begin{equation}\label{eq58}
  v((\nabla_X \varphi)U)=-\eta(U)\varphi(T_\xi X)+g(\varphi(T_\xi U),X)\xi.
\end{equation}
Applying Proposition \ref{Prop5} - \eqref{eq18} one has
\begin{equation*}
  \varphi(T_\xi U)=T_\xi \varphi U, \, U \in \chi^v(M).
\end{equation*}
Then, using well-known properties of $T$, one gets
\[\varphi \circ T_\xi=T_\xi \circ \varphi .\]

Furthermore, for any $X \in \chi^h(M),$ $U,V \in \chi^v(M)$, applying \eqref{eq54}, we have
\begin{equation*}
  g(T_U \varphi V-\varphi(T_UV),X)=g((\nabla_U \varphi)V,X)=-\eta(V)g(\nabla_{\varphi U}\xi, X)=-\eta(V)g(T_\xi \varphi U, X).
\end{equation*}
It follows $T_U\varphi V-\varphi(T_UV)=-\eta(V)T_\xi\varphi U$ and 4. is proved.

To prove 5., by Proposition \ref{Prop5} and the normality of $M$, one has
\begin{eqnarray*}
(\nabla_U \varphi)X&=&v((\nabla_U \varphi)X)=\frac{1}{2}N_\varphi (X, \varphi U)=-d \eta(X, \varphi U)\xi \\
&=& -\frac{1}{2}\big(g(T_\xi U, \varphi X)-g(T_\xi \varphi U,X)\big)\xi=g(T_\xi \varphi U, X)\xi.
\end{eqnarray*}
The last formula follows by Proposition \ref{Prop4} and \eqref{eq58}.
\end{proof}

\begin{remark}
Properties 1., 2. and 3. in Theorem \ref{theo2} were firstly obtained by Tshikuna-Matamba (Proposition 2.6, 2.7, 2.8, \cite{Tshi1}).
\end{remark}
\begin{example}\label{3.2}
Let $(M',J',g') $ be a K\"{a}hler manifold, $\dim M'=2n$, $(F, \hat{\varphi}, \hat{\xi}, \hat{\eta}, \hat{g})$ a quasi-Sasakian manifold, $\dim F=2r+1.$ It is easy to prove that the ACM-structure on the product manifold $M' \times F$ defined in Example \ref{3.1} is quasi-Sasakian. Therefore, the canonical projection $\pi:M\rightarrow M'$ is a contact-complex Riemannian submersion whose total space is a quasi-Sasakian manifold.

We also note that the $C_6$-components of $\nabla \phi$, $\hat{\nabla}\hat{\phi}$ are determined by $\delta \phi (\xi)$, $\hat{\delta}\hat{\phi}(\hat{\xi})$, respectively. Being $A=0$, one checks the equality \[\delta \phi (\xi)=\frac{n+r}{r}\hat{\delta}\hat{\phi}(\hat{\xi}).\]
It follows that, if $F$ is a $C_7$-manifold, then also $M $ falls in the class $C_7$, so $\pi: M \rightarrow M'$ provides an example of contact-complex Riemannian submersion whose domain is a $C_7$-manifold.
\end{example}
Now, we state the main properties of those submersions whose total space is a $C_6$-manifold.

The defining condition of class $C_6$ is expressed as:
\begin{equation}\label{new20}
 (\nabla_X \varphi)Y=\alpha(g(\varphi X,Y)\xi-\eta(Y)X),
\end{equation}
where $\alpha = \frac{\delta\phi(\xi)}{2m}$ is a smooth function. If $\alpha \neq 0$ is constant, then the manifold is called $\alpha$-Sasakian, Sasakian if $\alpha = 1.$ In \cite{Mar} the author proves that, if $m\geq 2,$ any connected $C_6$-manifold $M^{2m+1}$ either is $\alpha$-Sasakian or it is cosymplectic.

The main example of contact-complex submersion from a Sasakian manifold is the Hopf fibration $\pi: S^{2m+1}\rightarrow P_m(\mathbb{C}),$ \cite{Wat}, \cite{Fal}.
\begin{theorem}\label{new21}
Let $\pi: M \rightarrow M'$ be a contact-complex Riemannian submersion and assume that $M$ is a $\alpha$-Sasakian. Then, the following properties hold
\begin{itemize}
  \item[1. ] $M'$ is a K\"{a}hler manifold.
  \item[2. ] If $r \geq 1$, any fibre of $\pi$ is a minimal submanifold of $M$ and inherits from $M$ an $\alpha$-Sasakian structure.
  \item[3. ] The invariant $A$ acts as \[A_XY=-\alpha \phi(X,Y)\xi,\,\, X,Y \in \chi^h(M).\]
  \item[4. ] The invariant $T$ satisfies \[T_U \circ \varphi = \varphi \circ T_U, \,\,U \in \chi^v(M).\]
\end{itemize}
\end{theorem}
\begin{proof}
  Property 1. and the minimality of the fibres follow by Theorem \ref{theo2}. Applying \eqref{new20}, for any $U,V \in \chi^v(M)$ one has
  \begin{eqnarray*}
(\hat{\nabla}_U\varphi)V= v(({\nabla}_U\varphi)V)&=&\alpha (g(U,V)\xi-\eta(V)U),\\
T_U\varphi V - \varphi (T_UV)&=&h((\nabla_U\varphi )V)=0.
  \end{eqnarray*}
  It follows that the structure induced on any fibre is $\alpha$-Sasakian and the invariant $T$ satisfies
  \[(T_U \circ \varphi)_{/ \mathcal{V}}= (\varphi \circ T_U)_{/\mathcal{V}}, \,\, U \in \chi^v(M). \]
  Being each operator $T_U$ skew-symmetric, we obtain 4.

By \eqref{new20}, for any $X \in \chi(M)$ one has $\nabla_X \xi= -\alpha \varphi X$. By Proposition \ref{Prop1}, being $B=0$, for any $X,Y \in \chi^h(M)$ we obtain
\[A_XY=-g(\nabla_X \xi,Y)= \alpha g(\varphi X, Y)\xi.\]
\end{proof}
We end this section considering submersions whose total space falls in the class $C_{12}$. First, we recall the defining condition of any $C_{12}$-manifold, namely
\begin{equation}\label{eq59}
 (\nabla_X \varphi)Y=-\eta(X)\big((\nabla_\xi \eta)\varphi Y\xi+\eta(Y)\varphi(\nabla_\xi \xi)\big).
\end{equation}
The class $C_{12}$ is strictly related to the one of K$\ddot{a}$hler manifolds. More precisely, given a K$\ddot{a}$hler manifold $(M',J,g')$ and a smooth positive function $f:I \times M' \rightarrow \mathbb{R},$ $I$ being an open interval, on the product manifold $I \times M'$ one considers the ACM-structure $(\varphi , \xi , \eta , g_f)$ acting as:
\begin{equation*}
  \varphi(a \frac{\partial}{\partial t}, X)=(0,JX),\quad \eta( a \frac{\partial}{\partial t},X)=af,
\end{equation*}
\begin{equation*}
  \xi =\frac{1}{f}(\frac{\partial}{\partial t},0), \quad g_f=f^2\pi^*_1(dt \otimes dt)+\pi^*_2g',
\end{equation*}
for any $a \in \mathfrak{F}(I \times M'),$ $X \in \chi(M')$, $\pi_1: I \times M' \rightarrow I $,  $\pi_2 : I \times M' \rightarrow M'$ denoting the canonical projections.

The manifold $_fI \times M'=(I \times M', \varphi, \xi, \eta, g_f)$ falls in the class $C_{12}$, the projection $\pi_2 : _f\!I \times M' \rightarrow (M',J,g')$ is a contact-complex Riemannian submersion with 1-dimensional fibres. The action of the invariant $T$ is determined by
\begin{equation*}
  T_\xi \xi= h(\nabla_\xi \xi)=grad(\log f)-\xi (\log f)\xi,
\end{equation*}
the gradient being evaluated with respect to $g_f.$

It follows that $\pi_2$ has totally geodesic fibres if and only if the function $f=f(t,p)$ only depends on $t.$

Conversely, applying the theory developed in \cite{PR}, \cite{Fa3}, one proves that any $C_{12}$-manifold is, locally, realized as the ACM-manifold $_f]-\epsilon, \epsilon [ \times F$, $\epsilon >0$, $F$ being a K$\ddot{a}$hler manifold and $f: \,]-\epsilon, \epsilon [ \times F \rightarrow \mathbb{R}$ a smooth positive function. It follows that a contact-complex Riemannian submersion from a $C_{12}$-manifold with 1-dimensional fibres acts locally as the canonical projection onto a K$\ddot{a}$hler manifold.

Explicit examples of submersions from a $C_{12}$-manifold onto a K$\ddot{a}$hler one with fibres of dimension $2r+1 \geq 3$ are known (\cite{Tshi}).

Now,  we state the main properties of these submersions.
\begin{theorem}\label{thm3}
 Let $\pi: M\rightarrow M'$ be a contact-complex Riemannian submersion. Assume that $M$ is a $C_{12}$-manifold. The following properties hold:
 \begin{enumerate}
   \item $M'$ is a K$\ddot{a}$hler manifold.
   \item The horizontal distribution is totally geodesic.
   \item Any fibre $F$ inherits from $M$ a $C_{12}$-structure, provided that $\dim F \geq 3$.
   \item The invariant $T$ satisfies
   \begin{equation}\label{eq60}
     T_U \varphi V= \varphi(T_UV)-\eta(U)\eta(V)\varphi (T_\xi \xi),
   \end{equation}
   \begin{equation}\label{eq61}
     T_UV+T_{\varphi U} \varphi V=\eta(U)\eta(V)T_\xi \xi, \, U,V \in \chi^v (M).
   \end{equation}
   \item For any $ U \in \chi^v(M) $ , $X \in \chi^h(M)$ one has
   \begin{equation}\label{eq62}
     (\nabla_U \varphi )X=-\eta(U)g(T_\xi \xi,  \varphi X) \xi,
   \end{equation}
   \begin{equation}\label{eq63}
     (\nabla_X \varphi )U=0.
   \end{equation}
 \end{enumerate}
\end{theorem}

\begin{proof}
Properties 1. and 3. are known (\cite{Tshi}).

To state 2., remark that, being $M$ a $C_{12}$-manifold, one has
\begin{equation*}
  \nabla_X \xi = \eta(X)\nabla_\xi \xi , \, X \in \chi(M).
\end{equation*}
So, considering $U \in \chi^v(M)$, $X \in \chi^h(M)$, by Proposition \ref{Prop1}, one has
\begin{equation*}
  A_X U=\eta(U) A_X \xi =\eta(U) h(\nabla_X \xi)=0.
\end{equation*}
It follows that all the operators $A_X, \, X \in \chi^h(M)$, vanish and then $A=0.$ This implies that the horizontal distribution is totally geodesic.

Let $U,V$ be vertical vector fields. Applying \eqref{eq59}, we have
\begin{equation*}
  T_U \varphi V - \varphi (T_UV)= h((\nabla_U \varphi) V)=-\eta(U) \eta(V)\varphi (T_\xi \xi).
\end{equation*}
Therefore, we also obtain
\begin{equation*}
  T_U \varphi V =T_{\varphi U} V,\quad T_U \xi = \eta(U)T_\xi \xi ,
\end{equation*}
\begin{equation*}
  T_U V + T_{\varphi U}\varphi V=-T_U \varphi^2V+\eta(V)T_U \xi+T_{\varphi U}\varphi V=\eta(U)\eta(V) T_\xi \xi.
\end{equation*}
Moreover, applying well-known properties of $T$, for any $X \in \chi^h(M)$ we have $T_\xi X= \eta(T_\xi X)\xi.$ By Proposition \ref{Prop4}, \ref{Prop5} one gets
\eqref{eq63}.

Formula \eqref{eq62} follows by \eqref{eq59}.
\end{proof}

\begin{corollary}\label{newcor}
\textit{Let $\pi: M\rightarrow M'$ be a contact-complex Riemannian submersion. Assume that $M$ is a $C_{12}$-manifold and any fibre $(F^{2r+1},\hat{g})$ is a totally umbilical submanifold of $M$, $r \geq 1$. Then the invariant $T$ vanishes and $\pi$ is a totally geodesic map.}
\end{corollary}
\begin{proof}
  Let $(F, \hat{g})$ be any fibre of $\pi$. Applying Corollary \ref{cor1}, the mean curvature vector field of $F$ is \[\frac{1}{2r+1}T_\xi \xi \mid_F,\]
  and using the hypothesis, we have
  \begin{equation}\label{eq64}
    T_UV=\frac{1}{2r+1}\hat{g}(U,V) T_\xi \xi , \, U,V \in TF.
  \end{equation}
  Considering $U \in TF $, $U\neq 0$, $U \perp \xi$ by Theorem \ref{thm3}, we obtain
  \begin{eqnarray*}
g(U,U)T_\xi \xi &=& - g(U, \varphi ^2 U)T_\xi \xi = -(2r+1)T_U \varphi ^2 U\\
&=& -(2r+1)\varphi (T_U \varphi U)=0.
  \end{eqnarray*}
  It follows that $T_\xi \xi =0$ and by \eqref{eq64} we get that $(F, \hat{g})$ is totally geodesic. This implies $T=0$ and being $A=0$, $\pi$ is a totally geodesic map (Proposition 1.9, \cite{Fal}).
\end{proof}

Combining Corollary \ref{newcor} with a theorem of Vilms \cite{Vilms} one obtains.

\textit{If $\pi$ is a contact-complex Riemannian submersion with totally umbilical fibres and the total space $M$ is a complete, simply connected $C_{12}$-manifold, then $M$ is a Riemannian product and $\pi$ acts as the canonical projection on one of the factors.}

\section{The case $\omega \neq 0$} \label{section3}
The aim of this section is the study of contact-complex submersions in the case that the total space $M$ belongs to a suitable subclass of $C_2 \oplus C_4\oplus C_6\oplus C_7 \oplus \bigoplus_{9 \leq h \leq 12} C_h$ and the Lee form $\omega$, a priori, does not vanish. More precisely, we consider the following cases;
\begin{itemize}
  \item [\textbf{A)}] $M$ is locally conformal (l.c.) almost cosymplectic,
  \item[\textbf{B)}] $M$ is l.c. cosymplectic,
  \item[\textbf{C)}] $M$ is normal and $d\phi=-2\omega \wedge \phi$.
\end{itemize}

Firstly, we recall that an ACM manifold $(M,\varphi , \xi , \eta, g)$ is said to be l.c. almost cosymplectic (l.c. cosymplectic) if there exist an open covering $\{U_i\}_{i\in I}$ of $M$ and, for any $i$, a smooth function $\rho : U_i \rightarrow \mathbb{R}$ such that the local a.c.m. structure
\begin{equation}\label{eq66}
  \varphi_i= \varphi \mid_{U_i},\,\,\xi_i= \exp (-\rho_i)\xi\mid_{U_i},\,\, \eta_i=\exp(\rho_i)\eta\mid_{U_i},\,\,g_i=\exp(2\rho_i) g\mid_{U_i}
\end{equation}
is almost cosymplectic (cosymplectic).

The Lee form $\omega$ is related to the conformal change \eqref{eq66} by $\omega\mid_{U_i}=d\rho_i$ and, if $\dim M = 2m+1 \geq 5$, the following equivalences hold
\begin{itemize}
  \item [\textbf{I)}]
  $M$ is l.c. almost cosymplectic $\Leftrightarrow $ $d \phi =-2 \omega \wedge \phi,\,d\eta=\eta \wedge \omega$\\
    $\Leftrightarrow $ $M$ is a $C_2 \oplus C_4 \oplus C_5 \oplus C_9 \oplus C_{12}$-manifold and $\omega = \nabla_\xi \eta +\frac{\delta \eta}{2m} \eta$,

    \item [\textbf{II)}] $M$ is l.c. cosymplectic $\Leftrightarrow $ for any $X, Y, Z \in \chi(M)$ one has
        \begin{eqnarray}\label{eq400}
       (\nabla_X \phi)(Y,Z)&=& \omega (Y) \phi(X,Z)- \omega(Z) \phi(X,Y)\nonumber\\
       & &+\omega(\varphi Y)g(X,Z)-\omega(\varphi Z) g(X,Y)\\
       &\Leftrightarrow & M \text{ is a } C_4 \oplus C_5 \oplus C_{12}-\text{manifold and } \omega =\nabla_{\xi} \eta+\frac{\delta \eta}{2m} \eta.\nonumber
        \end{eqnarray}
\end{itemize}
By Proposition \ref{Prop1}, we obtain that if $\pi: M \rightarrow M'$ is a contact-complex Riemannian submersion and $M$ is l.c. almost cosymplectic (l.c. cosymplectic), then $M$ falls in the class $C_2\oplus C_4 \oplus C_9 \oplus C_{12}$ ($C_4 \oplus C_{12}$) and $\omega =\nabla_{\xi} \eta$, so that the Lee vector field is $B=\nabla_\xi \xi,\,\, h(B)=T_\xi \xi $.\\
We state the following results, dealing with the case \textbf{A)}.
\begin{theorem}\label{theo3.1}
  Let $\pi: M \rightarrow M'$ be a contact-complex Riemannian submersion. Assume that $\dim M =2m+1\geq 5 $, $\dim M' =2n \geq 4$ and $(M,\varphi , \xi , \eta, g)$ is l.c. almost cosymplectic. Then, one has
  \begin{itemize}
    \item[\textbf{i)}] $(M', J, g')$ is a locally conformal almost K\"{a}hler manifold.
    \item[\textbf{ii)}] The invariant $A$ acts as $A_XY=\phi(X,Y)\,v(\varphi B)$.
   \item[\textbf{iii)}] If $r=m-n \geq 2$, any fibre inherits from $M$ an l.c. almost cosymplectic structure. If $r=1$, then the fibres are $C_9\oplus C_{12}$-manifolds.
   \item[\textbf{iv)}] The vector field $h(B)=T_\xi \xi$ represents the mean curvature of the fibres.
  \end{itemize}
\end{theorem}
\begin{proof}
  By Proposition \ref{Prop3}, $(M',J,g')$ is a $W_2\oplus W_4$-manifold. Moreover, the Lee form $\beta$ of $M'$ is closed, $\omega $ being closed. It follows that $(M',J,g')$ is l.c. to a $W_2$-manifold, namely it is l.c. almost K\"{a}hler \cite{Va2}.

  Let $X,Y$ be horizontal vector fields. Then, we have \[\eta(A_XY)=\frac{1}{2}g([X,Y], \xi)=-d\eta(X,Y)=-(\eta \wedge \omega)(X,Y)=0.\]
 By Proposition \ref{Prop1}, we have $A_XY=\phi(X,Y)v(\varphi B)$.

 Now, assume $r=m-n \geq 2$ and consider a fibre $(F,\hat{\varphi}, \hat{\xi}, \hat{\eta}, \hat{g})$ of $\pi$. By Proposition \ref{Prop2}, $\hat{\omega}=\omega\mid_{TF}$ is the Lee form of $F$, $d\hat{\phi}=-2\hat{\omega}\wedge \hat{\phi}$ and, being $d \eta=\eta \wedge \omega$, we have $d \hat{\eta}=\hat{\eta} \wedge \hat{\omega}$. It follows that $F$ is l.c. almost cosymplectic.

 We recall that, in dimension $3$, the class of a.c.m. manifolds reduces to $C_5\oplus C_6 \oplus C_9 \oplus C_{12}$. Hence, if $r=1$, any fibre $(F,\hat{\varphi}, \hat{\xi}, \hat{\eta}, \hat{g})$ is in the class $C_6 \oplus C_9 \oplus C_{12}$ and the $C_6$-component of $\hat{\nabla}\hat{\phi}$ is determined by $\delta\hat{\phi}(\hat{\xi})$. We consider a local orthonormal frame $\{X_1,...,X_{2n},U_1,U_2,\xi\}$ defined in an open set $W \subset M$, such that $U_1,U_2$ are tangent to $F$. In $W$, by \textbf{ii)} we have
 \begin{eqnarray*}
\delta\phi(\xi)&=& \sum_{i=1}^{2n}g((\nabla_{X_i}\varphi)X_i,\xi)+\sum_{i=1}^{2}g((\nabla_{U_i}\varphi)U_i,\xi)\\
&=&\sum_{i=1}^{2n}g(A_{X_i}\varphi X_i,\xi)+\delta \hat{\phi }(\hat{\xi})=\delta \hat{\phi }(\hat{\xi}).
 \end{eqnarray*}
 So, since the $C_6$-component of $\nabla \phi$ vanishes, we have $\delta \hat{\phi}(\hat{\xi})=\delta \phi(\xi)=0$ and \textbf{iii)} is proved.

 Property \textbf{iv)} follows by \eqref{eq26-1}, since in this case $h(B)=T_\xi \xi$.
\end{proof}
\begin{corollary}\label{cor3.2}
\textit{Let $\pi: M \rightarrow M'$ be a contact-complex Riemannian submersion as in Theorem \ref{theo3.1}. Then, we have;}
\begin{itemize}
  \item[\textbf{i)}] \textit{The horizontal distribution is totally geodesic if and only if $\nabla_\xi \xi$ is horizontal.}
  \item[\textbf{ii)}] \textit{Each fibre of $\pi$ is a minimal submanifold of $M$ if and only if $\nabla_\xi \xi$ is vertical if and only if $T_\xi \xi=0.$}
\end{itemize}
\end{corollary}
An invariant submanifold $N$ of an a.c.m. manifold $(M,\varphi , \xi , \eta, g)$ is said to be superminimal if $\nabla_U \varphi =0,$ for any $U \in \chi(N)$. Obviously, a superminimal submanifold is minimal.

Let $\pi: M \rightarrow M'$ be a contact-complex Riemannian submersion. We say that $\pi$ has superminimal fibres if each fibre of $\pi$ is a superminimal submanifold of $M$, namely if and only if $\nabla_U \varphi =0$, for any $U \in \chi^v(M)$.
\begin{proposition}\label{prop3.3.3}
Let $\pi: M \rightarrow M'$ be a contact-complex Riemannian submersion as in Theorem \ref{theo3.1} and assume that $r=m-n \geq 1$. The following conditions are equivalent
\begin{itemize}
  \item[\textbf{i)}] $\pi$ has superminimal fibres,
  \item[\textbf{ii)}] each fibre inherits from $M$ a cosymplectic structure ant the invariant $T$ satisfies:
      \begin{equation*}
        T_U \varphi V= \varphi (T_UV),\, U,V \in \chi^v(M).
      \end{equation*}
\end{itemize}
\end{proposition}
\begin{proof}
For any $ U,V \in \chi^v(M)$, one has
\[h((\nabla_U \varphi)V)=T_U\varphi V-\varphi(T_UV),\]
\[v((\nabla_U \varphi)V)=(\hat{\nabla}_U \varphi)V,\]
and $\hat{\nabla}$ represents the Levi-Civita connection of the fibers. So, if $\pi$ has superminimal fibers, \textbf{ii)} holds.

Conversely, assuming \textbf{ii)}, we have for any $U,V \in \chi^v(M)$ $(\nabla_U \varphi)V=0.$ Let $X$ be horizontal. Then, by \eqref{eq16} and Theorem \ref{theo3.1}, we have
\begin{equation*}
  h((\nabla_U \varphi)X)=\eta(U)(A_{\varphi X}\xi-\varphi (A_X \xi))=0,\, U \in \chi^v(M).
\end{equation*}
On the other hand, we have $v((\nabla_U \varphi)X)=0, \,U\in \chi^v(M)$, since $g((\nabla_U \varphi)X,V)=-g((\nabla_U \varphi)V,X)=0, \, V \in \chi^v(M)$. It follows $\nabla_U\varphi=0, \, U \in \chi^v(M)$, and $\pi $ has superminimal fibres.
\end{proof}
In case \textbf{B)}, we state the next results dealing with the structure of the fibres and the behaviour of their second fundamental form.
\begin{theorem}\label{theo3.2}
  Let $\pi: M \rightarrow M'$ be a contact-complex Riemannian submersion, $\dim M=2m+1 \geq 5$, $\dim M'=2n+1 \geq 4$. If $(M,\varphi , \xi , \eta, g)$ is l.c. cosymplectic, then $(M', J, g')$ is l.c. K\"{a}hler, the invariant $T$ satisfies
  \begin{equation}\label{eq300}
    T_\xi U=\eta(U)T_\xi \xi, \quad T_\xi X=-\omega(X)\xi, \, U \in \chi^v(M),\, X \in \chi^h(M).
  \end{equation}
  Moreover, if $r=m-n \geq 2$, any fibre of $\pi$ is l.c. cosymplectic, if $r=1$, then the fibres are $C_{12}$-manifolds.
\end{theorem}
\begin{proof}
  Assume that $(M', J, g')$ is l.c. cosymplectic, so \eqref{eq400} holds. By direct calculation, applying \eqref{eq400} and Proposition \ref{Prop3}, the action of the covariant derivative $\nabla' \Omega$, $\Omega$ being the fundamental form of $M'$, is given by
  \begin{eqnarray*}
 (\nabla'_{X'}\Omega)(Y',Z')\circ \pi&=&\{\beta(Y')\Omega(X',Z')-\beta(Z')\Omega(X',Y')+\beta(JY')g'(X',Z')\\
 & &- \beta(JZ')g'(X',Y')\}
\circ \pi ,\,\,\, X',Y',Z' \in \chi(M').
\end{eqnarray*}
Note that the Lee form $\beta$ of $M'$ is closed, $\omega $ being closed. It follows that $(M',J,g')$ is an l.c. K\"{a}hler manifold.

Let $U$ be a vertical vector fields. By \eqref{eq400}, for any $Y \in \chi^h(M)$ we have
\begin{equation*}
  g(T_U \xi , \varphi Y)=g(\nabla_U \xi, \varphi Y)=(\nabla_U \phi)(Y,\xi)=\omega(\varphi Y)\eta(U)=\eta(U)g(T_\xi \xi , \varphi Y).
\end{equation*}
It follows $T_\xi U=T_U \xi=\eta(U)T_\xi \xi$. Thus, applying well-known properties of the operator $T_\xi$, we also obtain the second relation is \eqref{eq300}.

Let $(F,\hat{\varphi}, \hat{\xi}, \hat{\eta}, \hat{g})$ be a fibre of $\pi$. If $\dim F \geq 5$, by \eqref{eq400} and Proposition \ref{Prop1}, for any $U,V,W \in \chi(F)$, we have
\begin{equation*}
  (\hat{\nabla}_U \hat{\phi})(V,W)=\hat{\omega}(V)\hat{\phi}(U,W)-\hat{\omega}(W)\hat{\phi}(U,V)+\hat{\omega}(\hat{\varphi }V)\hat{g}(U,W)-\hat{\omega}(\hat{\varphi}W)\hat{g}(U,V).
\end{equation*}
It follows that $F$ is l.c. cosymplectic.

Assume that $\dim F =3$. By Theorem \ref{theo3.1} we know that $F$ is a $C_9 \oplus C_{12}$-manifold. Given a point $p \in F$, we consider an orthonormal frame $\{U_1, U_2=\varphi U_1, U_3=\hat{\xi}\}$ defined in an open neighbourhood $W$ of $p$. Since the Lee form of $F$ is $\hat{\omega}=\hat{\nabla}_{\hat{\xi}}\hat{\eta}=\omega/_{TF}$,
by \eqref{eq400} we easily obtain
\begin{equation*}
  (\hat{\nabla}_{U_i}\hat{\phi})(U_j , U_k)=\hat{\eta}(U_i)\bigg((\hat{\nabla}_{\hat{\xi}}\hat{\eta})\,\hat{\varphi}U_j \, \hat{\eta}(U_k)-\hat{\eta}(U_j)(\hat{\nabla}_{\hat{\xi}}\hat{\eta})\,\hat{\varphi}U_k\bigg)
\end{equation*}
for any $i,j,k \in \{1,2,3\}$. Hence, one has
\begin{equation*}
  (\hat{\nabla}_{U_i}\hat{\phi})(U_j)=-\hat{\eta}(U_i)\bigg(\hat{\eta}(U_j)\hat{\varphi}(\hat{\nabla}_{\hat{\xi}}\hat{\xi})+(\hat{\nabla}_{\hat{\xi}}\hat{\eta})\hat{\varphi}U_j\, \hat{\xi}\bigg),\,i,j \in \{1,2,3\}.
\end{equation*}
It follows that $\hat{\nabla}\hat{\varphi}$ satisfies \eqref{eq59}, so $(F,\hat{\varphi}, \hat{\xi}, \hat{\eta}, \hat{g})$ is a $C_{12}$-manifold.
\end{proof}

\begin{proposition}\label{prop3.5}
Let $\pi: M \rightarrow M'$ be a contact-complex Riemannian submersion as in Theorem \ref{theo3.2}. If $M$ is l.c. cosymplectic and $r=m-n \geq 1$, the following conditions are equivalent:
\begin{itemize}
  \item[\textbf{i)}] $\pi$ has superminimal fibers,
  \item[\textbf{ii)}] each fibre of $\pi$ is a minimal submanifold of $M$ and inherits from $M$ a cosymplectic structure.
\end{itemize}
\end{proposition}
\begin{proof}
  By \eqref{eq400}, for any $U,V \in \chi^v(M)$, $X \in \chi^h(M)$, we have
  \begin{eqnarray*}
g((\nabla_U \varphi)V,X)&=&\omega(X)\phi(U,V)+\omega(\varphi X)g(U,V)\\
&=&g(T_\xi \xi, X)\phi(U,V)-g(\varphi (T_\xi \xi),X)g(U,V).
  \end{eqnarray*}
  So, we obtain:
  \begin{eqnarray*}
T_U\varphi V-\varphi(T_UV)&=&h((\nabla_U \varphi)V)\\
&=&\phi(U,V)T_\xi \xi -g(U,V)\varphi(T_\xi \xi).
  \end{eqnarray*}
  Then, the statement follows by Proposition \ref{prop3.3.3} and Corollary \ref{cor3.2}.
\end{proof}
In case \textbf{C)}, let $(M^{2m+1}, \varphi, \xi, \eta , g)$, $m \geq 2$ be the total space of a contact-complex Riemannian submersion. Applying the theory developed in \cite{Chi-Gon}, \cite{Fa2} one has:

$M$ is normal and $d\phi =-2\omega \wedge \phi$ if and only if $M$ is a $C_4 \oplus C_6 \oplus C_7$-manifold if and only if $\nabla \phi$ acts as
\begin{eqnarray}\label{eq3.4}
(\nabla_X \phi)(Y,Z)&=& \omega(Y)\phi(X,Z)-\omega(Z)\phi(X,Y)+\omega(\varphi Y)g(\varphi X, \varphi Z)\nonumber\\
& &-\omega(\varphi Z) g(\varphi X,\varphi Y)+\eta(Y)(\nabla_{\varphi X}\eta)Z+\eta(Z)(\nabla_Y \eta)\varphi X
\end{eqnarray}
for any $X,Y,Z \in \chi(M)$.
\begin{theorem}\label{theo3.3}
  Let $\pi: M \rightarrow M'$ be a contact-complex Riemannian submersion, $\dim M=2m+1 \geq 5$, $\dim M'=2n$. Assume that $M$ is a $C_4 \oplus C_6\oplus C_7$-manifold. Then, the following properties hold
  \begin{itemize}
    \item[\textbf{i)}] $M'$ is a $W_4$-manifold.
    \item[\textbf{ii)}] The invariant $A$ satisfies
    \[A_{\varphi X}U=\varphi (A_XU),\, A_XY=A_{\varphi X}\varphi Y,\, X,Y \in \chi^h(M),\, U \in \chi^v(M).\]
    \item[\textbf{iii)}] The invariant $T$ satisfies
    \[T_\xi \circ \varphi =\varphi \circ T_\xi.\]
    \item[\textbf{iv)}] If $r=m-n\geq 2 $, then any fibre is a $C_4 \oplus C_6 \oplus C_7$-manifold, if $r=1$, the fibres fall in the class $C_6.$
  \end{itemize}
\end{theorem}
\begin{proof}
We remark
 that, if $n=1$, then $M'$ is a K\"{a}hler manifold, so \textbf{i)} holds. If $n \geq 2$, by Proposition \ref{Prop3}, $(M',J,g')$ is a $W_2 \oplus W_4$-manifold. We are going to prove that the almost complex structure $J$ is integrable.

Given $X', Y' \in \chi(M)$, let $X,Y$ be their horizontal lifts. By direct computation we have
\[N_J(X',Y')=\pi_*\big(N_\varphi (X,Y)+2d\eta(X,Y)\xi\big)=0.\]
It follows $N_J=0$, namely $J$ is integrable and $M'$ falls in the Gray-Hervella class $W_4.$

For any $X,Y \in \chi(M)$, by \eqref{eq3.4} we get
\[g(\nabla_X \xi ,Y)= (\nabla_X \phi)(\xi, \varphi Y)=(\nabla_{\varphi X}\,\eta)\varphi Y=g(\nabla_{\varphi X}\xi,\varphi Y)\]
and then we obtain
\begin{equation}\label{eq3.5}
  \varphi(\nabla_X \xi)=\nabla_{\varphi X} \xi, \, X \in \chi(M).
\end{equation}
In particular, considering $X, Y \in \chi^h(M)$, one gets:
\[A_{\varphi X}\xi=h(\nabla_{\varphi X}\xi)=\varphi(A_X \xi)\]
\[\eta(A_XY-A_{\varphi X}\varphi Y)=-g(A_X \xi , Y)+g(A_{\varphi X }\xi , \varphi Y)=0\]
Therefore, \textbf{ii)} follows by Proposition \ref{Prop1}.

Moreover, by \eqref{eq3.5} for any $U \in \chi^v(M)$, we have
\[\varphi (T_\xi U)= h(\varphi(\nabla_U \xi))=h(\nabla_{\varphi U}\xi)=T_\xi \varphi U.\]
It follows $(\varphi \circ T_\xi)_{ / \mathcal{V}} =(T_\xi \circ \varphi )_{/ \mathcal{V}}$ and being $T_\xi$ a skew-symmetric operator, we obtain \textbf{iii)}.

If $r \geq 2$ statement \textbf{iv)} is a consequence of \eqref{eq3.4}, recalling that the Lee form of any fibre $F$ is $\hat{\omega}=\omega \mid_{TF}$.

Now,  we assume $r=1$, and consider a fibre $(F, \hat{\varphi}, \hat{\xi}, \hat{\eta}, \hat{g})$. Let $\{U_1, U_2=\varphi U_1,U_3=\hat{\xi}\}$ be a local orthonormal frame defined in an open subset $W \subset F.$ Being, $\hat{\delta \hat{\eta}}=0$, one has \[(\hat{\nabla}_{U_1}\hat{\eta})U_1+(\hat{\nabla}_{U_2}\hat{\eta})U_2=0.\]
Moreover, \eqref{eq3.5} entails
\[(\hat{\nabla}_{U_2}\hat{\eta})U_2=\hat{g}(\hat{\nabla}_{\varphi U_1}\hat{\xi},\varphi U_1)=\hat{g}(\hat{\nabla}_{ U_1}\hat{\xi}, U_1)=(\hat{\nabla}_{U_1}\hat{\eta})U_1.\]
It follows
\begin{equation}\label{eq3.6}
  (\hat{\nabla}_{U_1}\hat{\eta})U_1=(\hat{\nabla}_{U_2}\hat{\eta})U_2=0.
\end{equation}
By \eqref{eq3.5} we also obtain
\begin{eqnarray}\label{eq3.7}
\hat{\delta}\hat{\phi}(\hat{\xi})&=& \hat{g}((\hat{\nabla}_{U_1}\hat{\varphi})U_1, \hat{\xi})+\hat{g}((\hat{\nabla}_{U_2}\hat{\varphi})U_2, \hat{\xi})\\
&=& -2\hat{g}(\hat{\nabla}_{U_1}\hat{\xi},U_2)=2\hat{g}(\hat{\nabla}_{U_2}\hat{\xi}, U_1).\nonumber
\end{eqnarray}
Thus, by direct computation, \eqref{eq3.6} and \eqref{eq3.7}, one proves that the only non-vanishing values for $(\hat{\nabla}_{U_i}\hat{\phi})(U_j , U_k), \, i,j,k \in \{1,2,3\}$, are
\begin{eqnarray*}
(\hat{\nabla}_{U_1}\hat{\phi})(U_1, U_3)&=&-(\hat{\nabla}_{U_1}\hat{\phi})(U_3, U_1)=(\hat{\nabla}_{U_2}\hat{\phi})(U_2, U_3)\\
&=&-(\hat{\nabla}_{U_2}\hat{\phi})(U_3,U_2)=-\frac{1}{2}\hat{\delta}\hat{\phi}(\hat{\xi}).
\end{eqnarray*}
It follows that $\hat{\nabla}\hat{\phi}$ satisfies
\[(\hat{\nabla}_{U}\hat{\phi})(V,W)=\frac{\hat{\delta}\hat{\phi}(\hat{\xi})}{2}\big(\hat{g}(U,W)\hat{\eta}(V)-\hat{g}(U,V)\hat{\eta}(W)\big),\]
hence $F$ is a $C_6$-manifold.
\end{proof}
\begin{proposition}\label{Prop3.3}
Let $\pi : M \rightarrow M'$ be a contact-complex Riemannian submersion as in Theorem \ref{theo3.3} and assume that $r=m-n \geq 1$. If the fibers of $\pi$ are totally umbilical submanifolds of $M$, then they are totally geodesic, the Lee vector field $B$ is vertical and $(M' , J' , g')$ is a K\"{a}hler manifold.
\end{proposition}
\begin{proof}
  Being $T_\xi \xi=0$, the Lee vector field $H=\frac{2r}{2r+1}h(B)$ represents the mean curvature vector field of the fibers. So, assuming that $\pi$ has totally umbilical fibres, for any $U,V \in \chi^v(M)$ we have \[T_UV=\frac{2r}{2r+1}g(U,V)h(B).\]
  Therefore, being $T_\xi \xi =0$, we obtain $h(B)=0$. It follows that $T=0$ and $B$ is vertical.

  Finally, if $\dim M'=2$, then $M'$ is a K\"{a}hler manifold. If $\dim M' \geq 4$, we know that $M'$ falls in the class $W_4$ and, applying \eqref{eq12}, we also have $\beta=0$. It follows that $M'$ is a K\"{a}hler manifold.
\end{proof}
\begin{example}
  Let $(F, \hat{\varphi}, \hat{\xi}, \hat{\eta}, \hat{g})$ be an almost cosymplectic manifold, $\dim F = 2r+1 \geq 3$, and $(M', J, g')$ an almost Hermitian manifold. Assume that $\dim M'=2,$ so the structure $(J,g')$ is K\"{a}hler. Let $f: M' \rightarrow \mathbb{R}$ be a non-constant smooth function such that $f > 0$ everywhere. We define an almost contact metric structure on $M' \times F$ putting
  \begin{eqnarray}
\varphi (X,U)&=&(JX,\hat{\varphi}U), \, \eta(X,U)=f \hat{\eta}(U), \,\xi= \frac{1}{f}\hat{\xi},\nonumber\\
g_f&=& \pi^*_1g'+f^2\pi^*_2\hat{g},
  \end{eqnarray}
  for any $X \in \chi(M')$, $U \in \chi (F)$, $\pi_1: M'\times F \rightarrow M'$, $\pi_2: M' \times F \rightarrow F$ denoting the canonical projections.

  Since $g_f$ is the warped product metric of $g'$, $\hat{g}$ by $f $, the projection \\$\pi_1: (M' \times F, g_f) \rightarrow (M', g')$ is a Riemanninan submersion with fibres isometric to $(F,\hat{g})$ and integrable horizontal distribution. The O'Neill invariant\\ $T$ acts as \[T_UV= -g_f(U,V)grad \log f, \, U,V \in \chi(F).\]

  We denote by $M'\times_f F$ the almost contact metric manifold $(M'\times_f F, \varphi, \xi, \eta, g_f)$.

  Firstly, we observe that the fundamental forms $\phi, \Omega, \hat{\phi}$ of $M' \times_f F$, $M'$, $F$ are related by
  \begin{equation*}
    \phi((X,U), (Y,V))=\Omega(X,Y)+f^2 \hat{\phi}(U,V).
  \end{equation*}
  Being both $\Omega, \hat{\phi}$ closed, one has \[d\phi=2d(\log f ) \wedge \phi.\]
  So, putting $\omega= -d(\log f)$, the fundamental form $\phi$ satisfies \eqref{eq7}. Then, $\omega$ is the Lee form of $M' \times F$, and the Lee vector field $B=-grad \log f$ is horizontal.

  Moreover, since $\hat{\eta}$ is closed, we have
  \[d\eta=df \wedge \hat{\eta}= -\eta \wedge d(\log f).\]
  We evaluate the vector field $\nabla_\xi \xi$, when $\nabla$ is the Levi-Civita connection of $M' \times_f F$. We recall that, being $F$ almost cosymplectic, one has $\hat{\nabla}_{\hat{\xi}} \hat{\xi}=0,$ so \[\nabla_\xi \xi = \frac{1}{f^2}T_{\hat{\xi}} \hat{\xi} = B.\]
  Then, $\omega = \nabla_\xi \eta$ is exact, $d\phi = -2 \omega \wedge \phi, \,\, d\eta= \eta \wedge \omega.$
  It follows that $M' \times_f F$, which falls in the class $C_2 \oplus C_4\oplus C_9\oplus C_{12},$ is conformal to an almost cosymplectic manifold. More precisely, according to \eqref{eq66}, the structure $(\varphi, \xi, \eta, g)$ is conformal to the almost cosymplectic structure
  \begin{eqnarray}\label{new88}
    \tilde{\varphi}&=&\varphi, \,\, \tilde{\xi}= \exp (\log f) \xi = \hat{\xi}, \,\, \tilde{\eta}=\exp (-\log f ) \eta= \hat{\eta},\nonumber\\
    \tilde{g}&=&\exp (-2\log \rho)g_f= \pi^*_1\bigg(\frac{1}{f^2}g'\bigg)+\pi^*_2\hat{g}.
  \end{eqnarray}
  So, $\pi_1: M' \times_f F \rightarrow M'$ is an example of the submersions considered in \textbf{A)}.

  Moreover, it is easy to prove that the structure defined by \eqref{new88} is cosymplectic, if $(F, \hat{\varphi}, \hat{\xi}, \hat{\eta}, \hat{g})$ is a cosymplectic manifold. In this case, the manifold $M' \times_f F$ is conformal to a cosymplectic manifold and falls in the class $C_4 \oplus C_{12}$, so the projection $\pi_1: M' \times_f F \rightarrow M'$ is an example of the submersions considered in \textbf{B)}.
\end{example}

\end{document}